\documentclass[12pt]{article}
\usepackage{amsmath, amsthm, amsfonts, amssymb}
\usepackage{mathrsfs}
\usepackage{bbm}
\usepackage{amssymb}
\usepackage{txfonts}
\usepackage{graphicx}
\usepackage{wrapfig}

\usepackage[T1]{fontenc}
\usepackage[utf8]{inputenc}
\usepackage{authblk}
 \usepackage{verbatim}

\newlength{\defbaselineskip}
\newcounter{marnote}

\newcommand{\setlinespacing}[1]%
           {\setlength{\baselineskip}{#1 \defbaselineskip}}

\theoremstyle{plain}
\newtheorem{theorem}{Theorem}[section]
\newtheorem{corollary}[theorem]{Corollary}
\newtheorem{lemma}[theorem]{Lemma}

\theoremstyle{definition}
\newtheorem{definition}{Definition}[section]
\theoremstyle{remark}
\newtheorem{remark}{Remark}[section]
\numberwithin{equation}{section}

\begin{document}

\title{Pointwise A Priori Estimates for Solutions to Some p-Laplacian Equations}

\author{Xiaoqiang Sun\footnote{\ School of Mathematics, Sun Yat-sen University, Guangzhou, 510275, China.}\, $^{\boxtimes}$ \qquad Jiguang Bao\footnote{\ School of Mathematical Sciences, Beijing Normal University, Laboratory of Mathematics and Complex Systems, Ministry of Education, Beijing 100875, China.
\par \ Emails: sunxq6@mail.sysu.edu.cn, jgbao@bnu.edu.cn
\par \ $\boxtimes$ Corresponding author}
 }

\date{}

\maketitle

\begin{abstract}
In this paper, we apply blow-up analysis to study pointwise a priori estimates for some p-Laplace equations based on Liouville type theorems. With newly developed analysis techniques, we first extend the classical results of interior gradient estimates for the harmonic function to that for the p-harmonic function, i.e., the solution of $\Delta_{p}u=0,\  x\in\Omega$. We then obtain singularity and decay estimates of the sign-changing  solution of Lane-Emden-Fowler type p-Laplace equation $-\Delta_{p}u=|u|^{\lambda-1}u, \ x\in\Omega$, which are then extended to the equation with general right hand term $f(x,u)$ with certain asymptotic properties. In addition, pointwise estimates for higher order derivatives of the solution to Lane-Emden type p-Laplace equation, in a case of $p=2$, are also discussed.
\end{abstract}

\section{Introduction}
p-Laplacian equation arises in many studies of non-linear phenomena, for instance, in mathematical modeling of non-Newtonian fluids in physics \cite{ref11,ref16}, and in the theory of quasi-regular, quasi-conformal mappings in geometry \cite{ref13,ref14,ref15}. The p-Laplacian equation represents an important type of quasi-linear equations. It has the following form
\begin{equation}\label{m-f}
-\Delta_{p}u=f(x,u),\qquad x\in\Omega
\end{equation}
where $\Omega$ is a
domain in $\mathbb{R}^{N}$, and
$$
\Delta_{p}u:=div(|\nabla u|^{p-2}\nabla u),\qquad 1<p<\infty
$$
is p-Laplace operator. When $p=2$, $\Delta_{p}u$ is Laplace operator. $\Delta_{p}u$ is degenerate elliptic when $p>2$ and has singularity when
$p<2$.  

 We call $u\in C^{1}(\Omega)$ as a weak solution of \eqref{m-f}, if
\begin{equation}\label{weaksf}
\int_{\Omega}|\nabla u|^{p-2}(\nabla
u,\nabla\varphi)dx=\int_{\Omega}f(x,u)\varphi dx,
\qquad\forall\varphi\in C_{c}^{\infty}(\Omega).
\end{equation}

In this paper we assume $f\in C^{1}(\mathbb{R}^{N})$. In the following text, a weak solution always refers to that defined above with $ C^{1,\alpha}_{loc}$ regularity.

\begin{remark}
Although weak solutions in other weaker forms could be defined as well, for example $u\in W^{1,m}_{loc}(\Omega)\cap L^{\infty}_{loc}(\Omega)$ \cite{ref11}, or $C^{1,\alpha}_{loc}$ weak solution \cite{ref6}, however, we can deduce that these definitions are equivalent to each other due to the following established conclusion on regularity \cite{ref3,ref7,ref11, ref12}: \emph{If $u\in C^{1}$ is a weak solution of \eqref{m-f}, then there exists $\beta\in(0,1)$ such that $u\in W^{2,2}_{loc}(\Omega) \cap C^{1,\beta}_{loc}(\Omega).$
}
\end{remark}

When $f$=0, Eq. \eqref{m-f} becomes 
\begin{equation}\label{p-harmonic}
-\Delta_{p}u=0,\qquad x\in\Omega.
\end{equation}
The solution $u$ of \eqref{p-harmonic} is called p-harmonic function. In the case of $p=2$, it is well-known harmonic function, satisfying the following Laplace equation
\begin{equation}\label{harmonic}
-\Delta u=0,\qquad x\in\Omega.
\end{equation}
For harmonic function, the following pointwise interior estimate of its gradient \cite{ref5} can be derived via Poisson integrative formula or mean vale theorem:
\begin{equation}\label{est-har}
|\nabla u(x)| \leq N\cdot\sup_{\Omega}|u|\cdot
dist^{-1}(x,\partial\Omega),\qquad  \forall x \in \Omega.
\end{equation}
A natural question is that whether the above gradient estimate for harmonic function can be extended to that for p-harmonic function? We will seek to answer it in this paper. 

On the other hand, when $f$ has the form of $u^{\lambda}$, then Eq. \eqref{m-f} becomes the following Lane-Emden type equation 
\begin{equation}\label{Lane-mden}
-\Delta_{p}u=u^{\lambda},\qquad x\in\Omega.
\end{equation}
J. Serrin and H. Zou \cite{ref11} proved the Liouville type theorems of \eqref{Lane-mden}, i.e. nonexistence of the nonnegative nontrivial solution on the whole space $\mathbb{R}^{N}$. Based on the Liouville type theorems, P. Polacik, P. Quittner and P. Souplet \cite{ref9} obtained the following singularity and decay estimates for the nonnegative weak solution of \eqref{Lane-mden} by applying the method of blow-up analysis:

\emph{Let $0<p-1<\lambda<\Lambda_{s}(p),$ and let $\Omega \neq \mathbb{R}^{N}$ be a domain,
then there exists a constant $C=C(p,N,\lambda)>0$ $($independent of $\Omega$ and $u), $
such that any $($ nonnegative weak $)$ solution $u$ of \eqref{Lane-mden} satisfies
\begin{equation}\label{guji2-1}
u(x)+|\nabla u(x)|^{\frac{p}{\lambda+1}}\leq C dist^{-\frac{p}{\lambda+1-p}}(x,\partial\Omega) ,\qquad \forall x \in \Omega
\end{equation}}
where,
\begin{equation}
\Lambda_{s}(p):=
\begin{cases}
\frac{N(p-1)+p}{n-p},\qquad\qquad if N>p\\
\infty, \qquad \qquad \qquad if N\leq p
\end{cases}
\end{equation}
is the corresponding Sobolev critical exponent (\cite{ref1,ref2}).

Again, we are curious about that whether the decay estimates \eqref{guji2-1} could be obtained for the following Lane-Emden-Fowler type p-Laplacian equation \cite{ref7}
\begin{equation}\label{bianhaoR}
-\Delta_{p}u=|u|^{\lambda-1}u, \qquad x \in \Omega.
\end{equation}

To answer the above two questions, we here develop new analytic techniques to get pointwise a priori estimates for the solution and/or gradient of some types of quasilinear equations by coupling blow up analysis with Liouville-type theorems. In this paper, we first extend the interior estimates for the gradient of harmonic function to p-harmonic function. Subsequently, we study the singularity and decay estimates of the sign-changing  solutions of the Lane-Emden-Fowler type p-Laplacian equation, $-\Delta_{p}u=|u|^{\lambda-1}u$. Furthermore, we also discuss the pointwise a priori estimates for higher order derivatives of solution to $-\Delta u=u^{\lambda},x\in\Omega$, i.e., the Lane-Emden type p-Laplacian equation in the case of $p=2$. Different from the estimates in the form of $L^{p}$ norm or $C^{\alpha}$ seminorm, the pointwise a priori estimates deduced in this paper have some interesting features: (1) $dist^{-1}(x,\partial\Omega)$ is explicitly expressed in the right hand of the estimates. (2) The constant $C$ in the right hand of the estimates depends only on the exponents in the equation ($p$, $N$ and/or $\lambda$), and is independent of domain $\Omega$ and solution $u$. In this sense the pointwise estimates are also universal.

\section{Main Results}
\subsection{Interior estimates of gradient of p-harmonic function}
\begin{theorem}\label{thm1}
Let $p\in(1,\infty)$, and let $\Omega$ be a bounded domain in $\mathbb{R}^{N}$ $(N\geq2)$. Then there exists a constant $C=C(p,N)>0$, such that any $C^{1}$ weak solution $u$ of
$$\Delta_{p}u=0,\qquad x\in\Omega,$$
satisfies
\begin{equation}\label{eqn1-1}
|\nabla u(x)|\leq C \sup_{\Omega}|u|\cdot dist^{-1}(x,\partial\Omega) ,\qquad \forall x \in \Omega.
\end{equation}
\end{theorem}

\begin{remark}
This result extends the classical interior estimates \eqref{est-har} of gradient of the harmonic function \cite{ref5}. 
\end{remark}

\subsection{Singularity estimates for Lane-Emden-Fowler type equation }
For Lane-Emden-Fowler type equation
\begin{equation}\label{bianhao}
-\Delta_{p}u=|u|^{\lambda-1}u, \qquad x\in\Omega,
\end{equation}
according to \eqref{weaksf}, we say $u\in C^{1}(\Omega)$ is a weak solution of \eqref{bianhao}, if
\begin{equation}\label{weaks}
\int_{\Omega}|\nabla u|^{p-2}(\nabla u,\nabla \varphi)dx=\int_{\Omega}|u|^{\lambda-1}u\varphi dx,\qquad \forall\varphi\in C_{c}^{\infty}(\Omega).
\end{equation}
A weak solution is called "stable" if the second variation of energy function of \eqref{bianhao} is nonnegative. More specifically, we have the following definition \cite{ref7}. 
\begin{definition}\label{Defstable}
Denote
$$
L_{u}(v,\varphi)=\int_{\Omega}|\nabla u|^{p-2}(\nabla v,\nabla \varphi)dx+\int_{\Omega}(p-2)|\nabla u|^{p-4}(\nabla u,\nabla v)\cdot(\nabla v,\nabla \varphi)-\lambda|u|^{\lambda-1}v\varphi dx.
$$
The weak solution $u\in C^{1}$ of \eqref{bianhao} is called stable, if
\begin{equation}\label{Lus1}
L_{u}(\varphi,\varphi)\geq0, \qquad \forall\varphi\in C_{c}^{1}(\Omega).
\end{equation}
Accordingly, the weak solution $u\in C^{1}$ of \eqref{bianhao} is called stable outside a compact set $K\subset\Omega$, if
\begin{equation}\label{Lus2}
L_{u}(\varphi,\varphi)\geq0, \qquad \forall\varphi\in C_{c}^{1}(\Omega \setminus K).
\end{equation}
\end{definition}
\begin{remark}
If $u$ has finite Morse index, then $u$ is stable outside a compact set.
\end{remark}
\begin{theorem}\label{thm2}
Let $p>2$, and let $\Omega\neq\mathbb{R}^{N}$ be an arbitrary domain of $\mathbb{R}^{N}$. If
\begin{equation}\label{index2}
\begin{cases}
p-1<\lambda<\infty,  \qquad\qquad \ \ if\ N\leq\frac{p(p+3)}{p-1},\\
p-1<\lambda<\Lambda_{c}(N,p), \qquad if\ N>\frac{p(p+3)}{p-1},\\
\end{cases}
\end{equation}
then there exists $C=C(p,N,\lambda)>0$ $($independent of $\Omega$ and $u$ $)$, 
such that any stable weak solution $u\in C^{1}(\Omega)$ of \eqref{bianhao} satisfies
\begin{equation}  \label{guji1}
|u(x)|+|\nabla u(x)|^{\frac{p}{\lambda+1}}\leq C dist^{-\frac{p}{\lambda+1-p}}(x,\partial\Omega) ,\qquad \forall x \in \Omega.
\end{equation}
\end{theorem}
\begin{remark}
This theorem expands the estimates of singularity and decay for nonnegative solution of Lane-Emden type equation in \cite{ref9} to that for sign-changing  solution of Lane-Emden-Fowler type equation, under the condition of "stability" of the weak solution.
\end{remark}

\begin{theorem}\label{thm3}
Let $p>2$, and let $\Omega\neq\mathbb{R}^{N}$ be an arbitrary domain of $\mathbb{R}^{N}$, if
\begin{equation}\label{index3}
\begin{cases}
p-1<\lambda<\infty,  \qquad\qquad  \ if N\leq p,\\
p-1<\lambda<\frac{N(p-1)+p}{N-p} ,\qquad if N>p,\\
\end{cases}
\end{equation}
then there exists $C=C(p,N,\lambda)>0$ (independent of $\Omega$ and $u$),
such that any  weak solution $u\in C^{1}(\Omega)$ of \eqref{bianhao} which is stable outside a compact set $K\subset\mathbb{R}^{N}$  satisfies
\begin{equation}\label{guji2}
|u(x)|+|\nabla u(x)|^{\frac{p}{\lambda+1}}\leq C dist^{-\frac{p}{\lambda+1-p}}(x,\partial\Omega) ,\qquad \forall x \in \Omega.
\end{equation}
\end{theorem}
\noindent
The conclusions in the above two theorems can be extended for the following equation:
\begin{equation}\label{bianhaof}
-\Delta_{p}u=f(u), \qquad x\in\Omega.
\end{equation}
\begin{corollary}\label{pf1}
Assume that the relations between $N,p,\lambda$ are the same as that in Theorem \ref{thm2}. Let $\Omega\neq\mathbb{R}^{N}$ be an arbitrary domain of $\mathbb{R}^{N}$, and let $f$ satisfy
\begin{equation}
\lim_{u\rightarrow\infty}|u|^{1-\lambda}u^{-1}f(u)=l\in(0,+\infty).
\end{equation}
If $u\in C^{1}(\Omega)$ is a weak solution of \eqref{bianhaof} and satisfies \eqref{Lus1},
then there exists $C=C(p,N,\lambda,f)>0$ (independent of $\Omega$ and $u$), such that
\begin{equation}\label{guji3}
|u(x)|+|\nabla u(x)|^{\frac{p}{\lambda+1}}\leq C(1+ dist^{-\frac{p}{\lambda+1-p}}(x,\partial\Omega)) ,\qquad \forall x \in \Omega.
\end{equation}
\end{corollary}

\begin{corollary}\label{pf2}
Assume that the relations between $N,p,\lambda$ are  the same as that in Theorem \ref{thm3}. Let $\Omega\neq\mathbb{R}^{N}$ be an arbitrary domain of $\mathbb{R}^{N}$. Assume that $f$ satisfies
\begin{equation}
\lim_{u\rightarrow\infty}|u|^{1-\lambda}u^{-1}f(u)=l\in(0,+\infty).
\end{equation}
If $u\in C^{1}(\Omega)$ is a weak solution of \eqref{bianhaof} and satisfies \eqref{Lus2},
then there exists $C=C(p,N,\lambda,f)>0$ $($independent of $\Omega$ and $u)$, such that
\begin{equation}\label{guji4}
|u(x)|+|\nabla u(x)|^{\frac{p}{\lambda+1}}\leq C(1+ dist^{-\frac{p}{\lambda+1-p}}(x,\partial\Omega)) ,\qquad \forall x \in \Omega.
\end{equation}
\end{corollary}

\subsection{Estimates of higher order derivatives}
\begin{theorem}\label{gj}

Let $n\geq 2$, $1<\lambda<\lambda_{s}$ , here,
\begin{equation}
\lambda_{s}:=
\begin{cases}
\frac{n+2}{n-2}, \quad if \ N>2,\\
 \infty, \quad if \ N=2,\\
\end{cases}
\end{equation}
and let $\Omega$ be an arbitrary domain in  $\mathbb{R}^{N}$, $\Omega\neq\mathbb{R}^n$. Then for any nonnegative integer $s$, there exists $C=C(p, N,\lambda,s)>0$ (independent of $\Omega$ and $ u$) such that any nonnegative smooth solution $u$ of
\begin{equation}\label{laplace}
-\Delta u=u^{\lambda},\qquad x\in\Omega,
\end{equation}
satisfies
\begin{equation}\label{higherestimate}
\sum_{r=0}^{s}|\nabla^{r} u(x)|^{\frac{\lambda-1}{(\lambda-1)r+2}}\leq C dist^{-1}(x,\partial\Omega) ,\qquad \forall x \in \Omega.
\end{equation}
\end{theorem}
\begin{remark}
The requirements for the regularity of solution actually only need  $u\in C^{s}(\Omega)$. What's more, in fact, the requirements for the regularity of solution can be relaxed by applying the regularity theory of elliptic equation.
\end{remark}

\begin{remark}
 For the equation with general right hand term  $f$ that satisfies corresponding asymptotic conditions, by using the same method one can obtain the estimates of higher order derivatives of the solutions. The details are omitted here.
\end{remark}

\begin{remark}
Compared to the singularity and decay estimates of the solutions to elliptic equations that have been obtained in previous studies (e.g., \cite{ref17}), our result presented here gives a more precise estimate for higher order derivatives of the solution to p-Laplacian equation.

\end{remark}

\section{Lemmas}

In this section , we first introduce Liouville theorem for p-harmonic function \cite{ref11}.

\begin{lemma}\label{Liouville0}
Let $u$ be a nonnegative weak solution of
$$\Delta_{p}u=0,\qquad x\in\mathbb{R}^{N}\qquad(n>p)$$
or $$\Delta_{p}u\leq0,\qquad x\in\mathbb{R}^{N}\qquad (n\geq2)$$
then $u$ is  constant.\\
\end{lemma}

For Lane-Emden-Fowler type equation \eqref{bianhao}, the following Liouville type theorems have been established in \cite{ref7}.
\begin{lemma}\label{Liouville1}
Let $u\in C^{1}(\mathbb{R}^{N})$ be a stable weak solution of \eqref{bianhao}, $p>2$, if
\begin{equation}\label{index}
\begin{cases}
p-1<\lambda<\infty,  \qquad \qquad  if N\leq\frac{p(p+3)}{p-1},\\
p-1<\lambda<\Lambda_{c}(N,p), \qquad if N>\frac{p(p+3)}{p-1},\\
\end{cases}
\end{equation}
here,
\begin{equation}
\Lambda_{c}(N,p)=\frac{[(p-1)N-p]^{2}+p^{2}(p-2)-p^{2}(p-1)N+2p^{2}\sqrt{(p-1)(N-1)}}{(N-p)[(p-1)N-p(p+3)]}.
\end{equation}
then $u\equiv0.$
\end{lemma}

\begin{lemma}\label{Liouville2}
Let $u\in C^{1}(\mathbb{R}^{N})$ be a weak solution of \eqref{bianhao}  which is stable outside of a compact set of $K$, $p>2$, if
\begin{equation}\label{index3}
\begin{cases}
p-1<\lambda<\infty,  \qquad\qquad  if N\leq p,\\
p-1<\lambda<\frac{N(p-1)+p}{N-p} , \qquad if N>p,\\
\end{cases}
\end{equation}
then $u\equiv0.$
\end{lemma}

Along with the above Liouville type theorems, the following "Doubling lemma" proposed in \cite{ref9} also plays  an important role in blow-up analysis for the proof in this paper.
\begin{lemma}\label{double lemma}
Let $(X,d)$ be a compact metric space and let $D\subset\Sigma\subset X$, where $D$ is non-empty and $\Sigma$ is closed. Set $\Gamma:=\Sigma\setminus D$. Let $M:D\rightarrow(0,\infty)$ be bounded on compact subset of $D$, and fix a real $k>0$. If $y\in D$ satisfies
$$ M(y)dist(y,\Gamma)>2k ,$$
then there exists $x\in D$ such that
$$ M(x)dist(x,\Gamma)>2k,\qquad M(x)\geq M(y),$$
and
$$M(z)\leq 2M(x),\qquad \forall z\in D\cap \overline{B}(x,k M^{-1}(x)).$$

\end{lemma}

In addition, we present here the following result regarding the convergence of weak stable solutions of the p-Laplacian equation, which will be used in proofs of Theorems \ref{thm2} and \ref{thm3} . 

\begin{lemma}\label{stable}
Let $\Omega_{k}\upuparrows\mathbb{R}^{N}, K$ be a compact subset in $\mathbb{R}^{N}$. Assume that $u_{k}$ are a set of weak solutions to \eqref{bianhao} on $\Omega_{k}$, $k=1,2,\cdots $, respectively and that $($as $k$ large enough$)$ $u_{k}$ is stable outside of $K$. If
$$
u_{k}\rightarrow u \qquad in\  C_{loc}^{1,\alpha}(\mathbb{R}^{N}),
$$
then $u$ is a weak solution of \eqref{bianhao} and is stable outside of $K$.
\end{lemma}

\section{Proofs}

In this section we apply the method of blow-up analysis to derive pointwise estimates. The method of blow-up is an efficient approach to deduce estimates in  partial differential equations, which has been applied by many scholars in, for example, \cite{ref8,ref9,ref10}. The corresponding Liouville type theorems ( e.g. \cite{ref4,ref5,ref7,ref8,ref11,ref17}) are important basis of this method. This method can be briefly described as follows: by proof of contradiction, assuming that an estimate (in terms of the distance to $\partial\Omega $) fails, we could construct an appropriate auxiliary function and use "doubling" property, then the sequence of violating solutions $u_{k}$ will be increasingly large along a sequence of points $x_{k}$, such that each $x_{k}$ has a suitable neighborhood where the relative growth of $u_{k}$ remains controlled. After appropriate rescaling, we can blow up the sequence of neighborhoods and pass to the limit to obtain a bounded solution of a limiting problem in the whole of $\mathbb{R}^{N}$ based on $L^{p}$ estimates or $C^{\alpha}$ estimates. By applying Liouville type theorems to the limiting equation we finally derive a contradiction.

\subsection{The proof of Theorem \ref{thm1}}
\begin{proof}[Proof of Theorem\ref{thm1}]
Under the same condition with that in Theorem \ref{thm1}, we just need to prove the following estimate: $\forall \epsilon \in (0, 1)$, there exists a constant $C=C(p, N)$, such that any $C^{1}$ weak solution $u$ of
$$\Delta_{p}u=0,\qquad x\in\Omega,$$
satisfies
\begin{equation}\label{eqn1}
|\nabla u(x)|\leq C \sup_{\Omega}|u|\cdot diam^{\epsilon}(\Omega) \cdot dist^{-(1+\epsilon)}(x,\partial\Omega) ,\qquad \forall x \in \Omega.
\end{equation}
Assume $u\geq0$, otherwise let
$$w=u+sup_{\Omega}|u|, $$ \\
then $w\geq0$, and $$\Delta_{p}w=0, \qquad x\in\Omega.$$
If for $w,$ \eqref{eqn1} holds, then for $u$ and $\forall x \in \Omega$ we have
\begin{align*}
|\nabla u(x)|&=|\nabla w(x)|\leq C \mathop{\sup_{\Omega}}|w|\cdot diam^{\epsilon}(\Omega)\cdot dist^{-1+\epsilon}(x,\partial\Omega)\\
&\leq C \mathop{\sup_{\Omega}}|u|\cdot diam^{\epsilon}(\Omega)\cdot dist^{-(1+\epsilon)}(x,\partial\Omega).
\end{align*}
Moreover, we assume $\mathop{\sup_{\Omega}}|u|\cdot diam^{\epsilon}(\Omega)=1$. Otherwise, let $$\widetilde{w}=\frac{u}{\mathop{\sup_{\Omega}}|u|\cdot diam^{\epsilon}(\Omega)}, $$
then $\mathop{\sup_{\Omega}}|\widetilde{w}|\cdot diam^{\epsilon}(\Omega)=1$,  and $$\Delta_{p}\widetilde{w}=0, \qquad x\in\Omega.$$
If for $\widetilde{w}$,  \eqref{eqn1} holds, then for $u$, by
  $$\frac{|\nabla u(x)|}{\mathop{\sup_{\Omega}}|u|\cdot diam^{\epsilon}(\Omega)}=|\nabla \widetilde{w}(x)|\leq C dist^{-(1+\epsilon)}(x,\partial\Omega),$$
we know that
  \begin{equation}
  |\nabla u(x)|\leq C \sup_{\Omega}|u|\cdot diam^{\epsilon}(\Omega)\cdot dist^{-1+\epsilon}(x,\partial\Omega) ,\qquad \forall x \in \Omega.\\
  \end{equation}
Therefore, we only need to prove that if $u$ satisfies
$\Delta_{p}u=0$ on $\Omega$, and $u\geq0, \sup_{\Omega}|u|\cdot diam(\Omega)=1$, then there holds   \begin{equation}\label{eqn2}
  |\nabla u(x)|\leq C \cdot dist^{-(1+\epsilon)}(x,\partial\Omega) ,\qquad \forall x \in \Omega.
  \end{equation}
If \eqref{eqn2} fails, then for $k=1,2,\cdots $, there exist sequences $\Omega_{k},u_{k},y_{k}\in\Omega_{k}$ and $\epsilon_{k} \in (0,1)$,
such that $u_{k}$ is nonnegative $C^{1}$ weak solution of $\Delta_{p}u=0 $ on $\Omega_{k}$, with $\displaystyle{\sup_{\Omega_{k}}u_{k}\cdot diam^{\epsilon_{k}}(\Omega_{k})=1}$,
and the function
$$M_{k}:=|\nabla u_{k}|^{\frac{1}{1+\epsilon_{k}}} $$
satisfies
$$M_{k}(y_{k})>2k dist^{-1}(y_{k},\partial\Omega_{k}).$$
By Lemma \ref{double lemma}, it follows that $\exists x_{k}\in\Omega_{k}$,  such that
$$M_{k}(x_{k})>2k dist^{-1}(x_{k},\partial\Omega_{k}), $$
and
$$M_{k}(z)\leq2M_{k}(x_{k}), \quad\forall z:|z-x_{k}|\leq k M_{k}^{-1}(x_{k}).$$
Let $\lambda_{k}:=M_{k}^{-1}(x_{k})$, and we rescale $u_{k}$ by setting
$$v_{k}(y):=k^{\epsilon_{k}}\lambda_{k}^{\epsilon_{k}} \cdot u_{k}(x_{k}+\frac{1}{k^{\epsilon_{k}}}\lambda_{k}y),\qquad \forall y:|y|\leq k. $$
Note  that as $|y|\leq k $,
$$|\frac{1}{k^{\epsilon_{k}}} \cdot \lambda_{k} y|\leq \frac{1}{k^{\epsilon_{k}}} \cdot \frac{1}{2k}dist(x_{k},\partial\Omega_{k})\cdot k^{1+\epsilon_{k}}=\frac{1}{2}dist(x_{k},\partial\Omega_{k}), $$
thus $x_{k}+\frac{1}{k^{\epsilon_{k}}} \cdot \lambda_{k} y \in\Omega_{k}$.
As such, $v_{k}$ is well defined.
Meanwhile we have
\begin{align*}\label{vk}
|v_{k}(y)|&=k^{\epsilon_{k}}\lambda_{k}^{\epsilon_{k}} \cdot u_{k}(x_{k}+\frac{1}{k^{\epsilon_{k}}}\lambda_{k}y)\\
&\leq k^{\epsilon_{k}} \cdot (\frac{1}{2k})^{\epsilon_{k}} \cdot dist^{\epsilon_{k}}(x_{k},\partial\Omega_{k})\cdot  \sup_{\Omega_{k}}u_{k}\\
&\leq (\frac{1}{2})^{\epsilon_{k}} \cdot  \sup_{\Omega_{k}}u_{k}\cdot diam^{\epsilon_{k}}(\Omega_{k})\\
&=(\frac{1}{2})^{\epsilon_{k}}\leq 1, \\
|\nabla v_{k}(y)|&=\lambda_{k}^{1+\epsilon_{k}} \cdot |\nabla u_{k}(x_{k}+\frac{1}{k^{\epsilon_{k}}}\lambda_{k}y)|\\
&=M_{k}(x_{k})^{-(1+\epsilon_{k})}\cdot M_{k}(x_{k}+\frac{1}{k^{\epsilon_{k}}}\lambda_{k}y)^{1+\epsilon_{k}}\leq 4, \qquad  \forall y:|y|\leq k\\
|\nabla v_{k}(0)|&=1,
\end{align*}
and
\begin{equation}\label{p-vk}
\Delta_{p}v_{k}(y)=div(|\nabla v_{k}|^{p-2}\cdot\nabla v_{k})=(\lambda_{k}^{1+\epsilon_{k}})^{p-1}\cdot\Delta_{p}u_{k}=0,
\end{equation}
that is, $v_{k}$ solves
$$\Delta_{p}v_{k}=0,\qquad x\in B_{k}(0).$$
By using $C^{\alpha}$ estimates, we deduce that
there exist $ \beta\in(0,1)$ and constant $C$ (independent of $k$) such that
$$\parallel v_{k}\parallel_{C_{loc}^{1+\beta}}\leq C.$$
Therefore there exist a subsequence of $ { v_{k} } $, still denoted by $ { v_{k} }$, such that
\begin{equation}\label{jixian}
v_{k}\rightarrow v \qquad in\ C_{loc}^{1}(\mathbb{R}^{N}),
\end{equation}
and at point  $0$, $v$ satisfies
\begin{equation}\label{dv0}
|\nabla v(0)|=1.
\end{equation}
By \eqref{p-vk} it follows that $v$ satisfies
$$\Delta_{p}v=0, \qquad x\in \mathbb{R}^{N}.$$
By Liouville theorem (Lemma \ref{Liouville0}) we deduce that
$$v=const, \qquad x\in \mathbb{R}^{N}.$$
So $|\nabla v|\equiv0$  in $\mathbb{R}^{N}$, this contradicts with \eqref{dv0}. Therefore \eqref{eqn2} holds and we have proved Theorem \ref{thm1}.

\end{proof}

\begin{remark}
In the above proof, we used an approximation method to get the estimate by first relaxing the inequality through introducing $\epsilon$ in \eqref{eqn1} and then taking limit, which is a critical step since we found that it was hard to apply blow up analysis for the estimate \eqref{eqn1-1} directly.  

\end{remark}

\subsection{The proof of Theorem \ref{thm2}}

\begin{proof}[Proof of Lemma \ref{stable}]
$\forall\varphi\in C_{c}^{1}(\mathbb{R}^{N}\setminus K)$, let $V=supp \varphi\subset\mathbb{R}^{N}\setminus K$,  $V\subset\Omega_{k}$ as $k$ large enough. We first prove that $u$ satisfies \eqref{weaks}. In fact, as $k$  large enough,
\begin{align*}
&|\int_{V}|\nabla u_{k}|^{p-2}(\nabla u_{k},\nabla\varphi) dx-\int_{V}|\nabla u|^{p-2}(\nabla u,\nabla\varphi) dx|\\
&\leq\int_{V}(|\nabla u_{k}|-|\nabla u|)^{p-2}(\nabla u_{k},\nabla\varphi) dx+ \int_{V}|\nabla u|^{p-2}(\nabla u_{k}-\nabla u,\nabla\varphi) dx  
\rightarrow0, \qquad (k\rightarrow\infty).
\end{align*}

We next prove
\begin{equation}
L_{u_{k}}(\varphi,\varphi)\rightarrow L_{u}(\varphi,\varphi)  \qquad (k\rightarrow\infty).
\end{equation}
In fact, as $k$  large enough,
\begin{align*}
&|\int_{V}|\nabla u_{k}|^{p-2}|\nabla\varphi|^{2} dx-\int_{V}|\nabla u|^{p-2}|\nabla\varphi|^{2} dx|\\
&\leq \displaystyle\max_{V}(|\nabla u_{k}|^{p-2}|-|\nabla u|^{p-2}|)\int_{V}|\nabla\varphi|^{2} dx \rightarrow0, \qquad (k\rightarrow\infty).
\end{align*}
Similarly,
$$
\int_{V}|u_{k}|^{\lambda-1}\varphi^{2}dx\rightarrow\int_{V}|u|^{\lambda-1}\varphi^{2}dx, \qquad (k\rightarrow\infty).
$$
Moreover,
\begin{align*}
&|\int_{V}|\nabla u_{k}|^{p-4}|(\nabla u_{k},\nabla\varphi)|^{2} dx-\int_{V}|\nabla u|^{p-4}|(\nabla u,\nabla\varphi)|^{2} dx|\\
&\leq\int_{V}(|\nabla u_{k}|-|\nabla u|)^{p-4}|(\nabla u_{k},\nabla\varphi)|^{2} dx+ \int_{V}|\nabla u|^{p-4}(\nabla u_{k}+\nabla u,\nabla\varphi)(\nabla u_{k}-\nabla u,\nabla\varphi) dx\\
&\rightarrow0\qquad (k\rightarrow\infty).
\end{align*}
Therefore,
\begin{equation}
L_{u_{k}}(\varphi,\varphi)\rightarrow L_{u}(\varphi,\varphi)\geq0 \qquad (k\rightarrow\infty).
\end{equation}
So $u$ is a weak solution to \eqref{bianhao} and is stable outside of $K$ according to the Definition \ref{Defstable}.
\end{proof}

\begin{proof}[Proof of Theorem \ref{thm2}]
Let $\alpha=\frac{p}{\lambda+1-p}$. Assume that estimate \eqref{guji2} fails, then there exist sequences $\Omega_{k},u_{k},y_{k}\in\Omega_{k}$ $(k=1,2,\cdots)$, such that each $u_{k}$ is a stable weak solution of \eqref{bianhao} on $\Omega_{k}$ and functions
$$M_{k}:=|u_{k}|^{\frac{1}{\alpha}}+|\nabla u_{k}|^{\frac{1}{(\alpha+1)}}, \qquad k=1,2,\cdots,$$
satisfy $$M_{k}(y_{k})>2k dist^{-1}(y_{k},\partial\Omega_{k}).$$
By Lemma \ref{double lemma} it follows that $\exists x_{k}\in\Omega_{k}$, such that
$$M_{k}(x_{k})>2k dist^{-1}(x_{k},\partial\Omega_{k}), $$
$$M_{k}(z)\leq2M_{k}(x_{k}), \quad\forall z:|z-x_{k}|\leq k M_{k}^{-1}(x_{k}).$$
Let $\lambda_{k}:=M_{k}^{-1}(x_{k})$, and rescale $u_{k}$ by setting
$$v_{k}(y):=\lambda_{k}^{\alpha}u_{k}(x_{k}+\lambda_{k}y),\qquad \forall y:|y|\leq k. $$
Note that $(p-1)(\alpha+1)+1=\lambda\alpha$. It's easy to verify that $v_{k}$ is a stable weak solution of
\begin{equation}\label{vkbianhao}
-\Delta_{p}v_{k}=|v_{k}|^{\lambda-1}v_{k}, \qquad y\in B_{k}(0)
\end{equation}
and
\begin{equation}\label{vkbianhao0}
[|v_{k}|^{\frac{1}{\alpha}}+|\nabla v_{k}|^{\frac{1}{(\alpha+1)}}](0)=1,\qquad \qquad
\end{equation}

\begin{equation}\label{vkbianhao2}
[|v_{k}|^{\frac{1}{\alpha}}+|\nabla v_{k}|^{\frac{1}{(\alpha+1)}}](y)\leq2,\qquad|y|\leq k.
\end{equation}
By using $C_{loc}^{1,\beta}$ estimates of p-Laplacian equation, we know $\exists \beta\in(0,1)$ and $C$  (independent of $k$$)$ such that
$$\parallel v_{k}\parallel_{C_{loc}^{1+\beta}}\leq C, \qquad \forall k=1,2,\cdots.$$
By Arzela-Ascolli Theorem, it follows that there exist a subsequence of ${v_{k}}$, still denoted by ${v_{k}}$, such that

\begin{equation}\label{jixian2}
v_{k}\rightarrow v \qquad in\ C_{loc}^{1}(\mathbb{R}^{N}).
\end{equation}
and by \eqref{vkbianhao} we deduce that $v$ satisfies
\begin{equation}\label{v22}
\Delta_{p}v=|v|^{\lambda-1}v, \qquad x\in \mathbb{R}^{N}.
\end{equation}
By Lemma \ref{stable}, $v$ is a stable weak solution of $\eqref{v22}$. By Lemma \ref{Liouville1}
we know that
$$v=0, \qquad x\in \mathbb{R}^{N}.$$
But by \eqref{vkbianhao0} it follows that
\begin{equation}
[|v|^{\frac{1}{\alpha}}+|\nabla v|^{\frac{1}{(\alpha+1)}}](0)=1.
\end{equation}
Contradiction is derived. As such we have proved Theorem \ref{thm2}.
\end{proof}

\begin{proof}[Proof of Theorem \ref{thm3}]
By applying the same procedure in  proof of Theorem \ref{thm2}, it is easy to deduce the conclusion by Lemma \ref{Liouville2} and Lemma \ref{stable}.
\end{proof}

\begin{proof}[Proofs of Corollary \ref{pf1} and Corollary \ref{pf2}]
Similar to the proofs of Theorem \ref{thm2} and Theorem \ref{thm3}, the difference lies in that
$$M_{k}(y_{k})>2k(1+dist^{-1}(y_{k},\partial\Omega_{k}))>2k dist^{-1}(y_{k},\partial\Omega_{k}), $$
and $$\lambda_{k}\rightarrow0 \quad (k\rightarrow\infty). $$\\
$v_{k}$ solves
\begin{equation}\label{evkf}
-\Delta_{p}v_{k}(y)=f_{k}(v_{k}(y)):=\lambda_{k}^{(\alpha+1)(p-1)+1}f(\lambda_{k}^{-\alpha v_{k}(y)}), \qquad|y|\leq k.
\end{equation}
By using $C^{\alpha}$ estimates, we deduce that there exist $ \beta\in(0,1)$ and constant $C ($independent of $k$$)$ such that
$$\parallel v_{k}\parallel_{C_{loc}^{1+\beta}}\leq C, \qquad \forall k=1,2,\cdots.$$
Therefore there exist a subsequence of ${v_{k}}$, still denoted by ${v_{k}}$, such that
\begin{equation}\label{jixian2}
v_{k}\rightarrow v \qquad in\ C_{loc}^{1}(\mathbb{R}^{N}),
\end{equation}
and by \eqref{evkf} we deduce that $v$ satisfies \eqref{vkbianhao0}-\eqref{vkbianhao2} and solves
\begin{equation}
\Delta_{p}v=l\cdot|v|^{\lambda-1}v, \qquad x\in \mathbb{R}^{N}.
\end{equation}
By Liouville theorems $($Lemma \ref{Liouville1} and Lemma \ref{Liouville2}$)$ we know
$$v=0, \qquad x\in \mathbb{R}^{N}.$$
The contradiction is also derived.
\end{proof}

\begin{proof}[Proofs of Theorem \ref{gj}]
Assume that \eqref{higherestimate} fails, then, there exist sequences $\Omega_{k},u_{k},y_{k}\in\Omega_{k}$ $(k=1,2,\cdots)$,
such that each $u_{k}$ is a smooth solution of $-\Delta u=u^{\lambda} $ on $\Omega_{k}$, and functions
\begin{equation}
M_{k}:=\sum_{r=0}^{s}|\nabla^{r} u_{k}|^{\frac{1}{\alpha+r}}, \qquad k=1,2,\cdots,
\end{equation}
satisfy $$M_{k}(y_{k})>2k dist^{-1}(y_{k},\partial\Omega_{k}),$$\\
where $\alpha=\frac{2}{\lambda-1}$.\\
By Lemma \ref{double lemma} it follows that there exists $ x_{k}\in\Omega_{k}$,  such that
$$M_{k}(x_{k})>2k dist^{-1}(x_{k},\partial\Omega_{k}), $$
$$M_{k}(z)\leq2M_{k}(x_{k}), \quad\forall z:|z-x_{k}|\leq k M_{k}^{-1}(x_{k}).$$
Let $\lambda_{k}:=M_{k}^{-1}(x_{k})$, and rescale $u_{k}$ by setting
$$v_{k}(y):=\lambda_{k}^{\alpha}u_{k}(x_{k}+\lambda_{k}y),\qquad \forall y:|y|\leq k. $$
Note that as $|y|\leq k$ we have
$$|\lambda_{k} y|\leq\frac{1}{2k}dist(x_{k},\partial\Omega_{k})\cdot k=\frac{1}{2}dist(x_{k},\partial\Omega_{k}). $$
This implies that $$x_{k}+\lambda_{k}y\in\Omega_{k} , $$
Therefore $v_{k}$ is well defined. We further deduce that
\begin{align*}\label{mvk2}
-\Delta v_{k}(y)&=-\lambda_{k}^{\alpha}\cdot \Delta_{y}u_{k}(x_{k}+\lambda_{k}y)\\
&=-\lambda_{k}^{\alpha+2}\Delta u_{k}\\
&=\lambda_{k}^{\alpha+2} u_{k}^{\lambda}=(\lambda_{k}^{\alpha} u_{k})^{\lambda}\\
&=v_{k}^{\lambda}(y).
\end{align*}
So $v_{k}$ solves
$$-\Delta v_{k}=v_{k}^{\lambda},\qquad y\in B_{k}(0),$$
and
\begin{align*}
\sum_{r=0}^{s}|\nabla^{r}v_{k}(y)|^{\frac{1}{\alpha+r}}&=\sum_{r=0}^{s}|\lambda_{k}^{\alpha+r}\nabla^{r}u_{k}|^{\frac{1}{\alpha+r}}\\
&=\lambda_{k}M_{k}(x_{k}+\lambda_{k}y)\\
&\leq2,  \qquad\qquad\qquad |y|\leq k;
\end{align*}
\begin{equation}\label{2dvk0}
\sum_{r=0}^{s}|\nabla^{r}v_{k}|^{\frac{1}{\alpha+r}}(0)=\lambda_{k}M_{k}(x_{k})=1.
\end{equation}
By using $C^{\alpha}$ estimates, we deduce that
there exist $\beta\in(0,1)$ and constant $C$ (independent of $k$) such that
$$\parallel v_{k}\parallel_{C_{loc}^{s+\beta}}\leq C.$$
Therefore there exist a subsequence of ${v_{k}}$, still denoted by ${v_{k}}$, such that
\begin{equation}\label{jixian}
v_{k}\rightarrow v \qquad in \ C_{loc}^{s}(\mathbb{R}^{N}),
\end{equation}
and $v$ solves
\begin{equation}
-\Delta v=v^{\lambda},  \qquad x\in \mathbb{R}^{N}.
\end{equation}
By Lemma \ref{Liouville0} we deduce
$$v=const, \qquad x\in \mathbb{R}^{N}.$$
This contradicts with \eqref{2dvk0}, thus we have proved Theorem \ref{gj}.
\end{proof}

\noindent{\bf{\large Acknowledgements.}}
This study was supported by grants from the National Natural Science Foundation of China (11871070), the Guangdong Basic and Applied Basic Research Foundation (2020B151502120), the Fundamental Research Funds for the Central Universities (20ykzd20).

\bibliographystyle{unsrt}

\bibliography{references}

\end{document}